\documentclass[]{article}
\usepackage{indentfirst}
\usepackage{latexsym}
\usepackage{amsfonts}
\usepackage{amssymb}
\usepackage{amsmath}
\usepackage[noend]{algpseudocode}
\usepackage{algorithmicx,algorithm}
\usepackage{graphicx}
\usepackage{bm}
\usepackage{setspace}
\usepackage{graphicx}
\usepackage{amsthm}
\usepackage{geometry}
\usepackage{float}
\makeatother
\usepackage{subfigure}
\usepackage{multirow}
\usepackage{makecell}
\usepackage{longtable}
\usepackage{booktabs}
\usepackage{caption}

\theoremstyle{plain}
\newtheorem{theorem}{Theorem}[section]

\newtheorem{corollary}[theorem]{Corollary}
\newtheorem{lemma}[theorem]{Lemma}
\newtheorem{proposition}[theorem]{Proposition}

\theoremstyle{definition}
\newtheorem{definition}{Definition}[section]
\begin{document}
\title{\bf {Geometric Characteristics of Wasserstein Metric on SPD(n)}}
\author{
Yihao Luo, Shiqiang Zhang, Yueqi Cao, Huafei Sun \\
{\small \it School of Mathematics and Statistics, Beijing Institute of Technology, Beijing 100081, P. R. China}\\
}
\date{}
\maketitle{}

\noindent\textbf{Abstract}: Wasserstein distance, especially among symmetric positive-definite matrices, has broad and deep influences on development of artificial intelligence (AI)  and other branches of computer science. A natural idea is to describe the geometry of $SPD\left(n\right)$ as a Riemannian manifold endowed with the Wasserstein metric. In this paper, by involving the fiber bundle, we obtain explicit expressions for some locally geometric quantities, including geodesics, exponential maps, the Riemannian connection, Jacobi fields and curvatures. Furthermore, we discuss the behaviour of geodesics and prove that the manifold is globally geodesic convex with non-negative curvatures but no conjugate pair and cut locus. According to arithmetic estimates, we find curvatures can be controlled by the minimal eigenvalue.

\noindent\textbf{Key Words}:   Symmetric positive-definite matrix; Wasserstein distance; geodesic; connection; curvature

\section{Introduction}
\subsection{Background}

The importance of symmetric positive-definite matrices \cite{ref21, ref12} is beyond words. Symmetric positive-definite matrices have wide usage in many fields of information science, such as stability analysis of signal processing, linear stationary
systems, optimal control strategies, and imaging analysis \cite{ref11, ref5}. Due to the importance of  covariance matrices, analyses about symmetric positive definite matrices induces series of more profound discussions for multivariate statistics\cite{ref14}. Instead of considering a single matrix, contemporary scientists tend to comprehend the global structure of the set consisting of all $n\times n$ symmetric positive definite matrices. This set is known as  $SPD\left(n\right)$.

The set $SPD\left(n\right)$ can be endowed with various structures. $SPD\left(n\right)$ has a natural differential structure to be a manifold since it can be regarded as an open subset in $\mathbb{R}^{\frac{n\times\left(n+1\right)}{2}}$. It is also a quotient space of the general linear group $GL\left(n\right)$ by the orthogonal group ${O}\left(n\right)$  \cite{ref17}, denoted as $SPD\left(n\right)\cong\frac{GL\left(n\right)}{{O}\left(n\right)}$. Further, in the sense of matrix exponential and matrix logarithm \cite{ref12}, $SPD\left(n\right)$ can be regard as a Lie group \cite{ref18}. The group multiplication is defined by ${A}_1\ast {A}_2 = \exp\left(\log{{A}_1}+\log{{A}_2}\right)$ \cite{ref15}. Various structures bring abundant geometric details. And from another sight, the family of $n$-dimensional Gaussian distributions can be equipped with distances and divergences. These measurements can be naturally carried onto $SPD\left(n\right)$. Therefore, there are some widely used measurements defined on $SPD\left(n\right)$, including Kullback-Leibler (KL) divergence, Jensen-Shannon (JS) divergence, and Wasserstein distance \cite{ref16}. These measurements all contain deep physical background meanings and are extensively used in AI and data science. To balance above two ideas, scientists are trying to find some Riemannian metrics on $SPD\left(n\right)$. Series of metrics have been involved. The most traditional Euclidean metric implies the Euclidean distance among Gaussian distributions. X. Pennec, P. Fillard, et al.  \cite{ref17} defined the affine-invariant Riemannian metric. V. Arsigny, P. Fillard, et al. \cite{ref15} showed the Lie group structure admitting a bi-invariant metric, which is called the Log-Euclidean metric. And most encouragingly, by constructing a principle bundle, M. Wong et al. \cite{ref22} and  S. Zhang et al. \cite{ref8} gave a new Riemannian metric on $SPD\left(n\right)$ whose geodesic distance is equivalent to Wasserstein-$2$ distance.

Wasserstein distance  \cite{ref16} plays a key role in recent development of information science. This distance, called Earth-moving distance as well, describes the minimal price or energy required to move one distribution to another along an underlying manifold. Wasserstein distance is actually a family of distances depending on different definitions of energy:
\begin{equation}\label{Wsst}
	W_p\left(P_1,P_2\right)= \inf_{\gamma\sim\Pi\left(P_1,P_2\right)} \left(E_{\left(x,y\right)\sim\gamma}[\|x-y\|^p] \right)^{\frac{1}{p}}.
\end{equation}
These distances, especially Wasserstein-$2$ distance $W_2$ ($p$ is chosen as $2$ in followings), naturally reflect the metric structure of underlying manifold.
Wasserstein distance has its deep and profound physical meanings. And compared to KL divergence and JS divergence, Wasserstein distance can measure the difference between two distributions with small overlapped supports. Therefore Wasserstein distance has been regarded as an essential tool to access more advanced or explainable AI. For instance,  A. Martin et al. solved the problems about the robustness of nets training by introducing Wasserstein distance into Generative Adversarial Networks (GAN) \cite{ref2}. Even so, the general Wasserstein distance can not be expressed explicitly, unless on some specific distribution families.
Fortunately, the distance among Gaussian distributions on $\mathbb{R}^n$ has a beautiful explicit form  \cite{ref9}
\begin{equation}\label{Wss2d}
	W\left(\mathcal{N}_1,\mathcal{N}_2\right)=\|\mu_1-\mu_2\|+\left( {\rm tr}[\Sigma_1+\Sigma_2-2\left(\Sigma_1\Sigma_2\right)^{\frac{1}{2}}] \right)^\frac{1}{2},
\end{equation}
where $\mu_1$, $\mu_2$ and $\Sigma_1$, $\Sigma_1$ are expectations and covariances of distributions $\mathcal{N}_1$, $\mathcal{N}_2$ respectively. Because Wasserstein distance is originally defined by a low bound of energy, whether a real moving process among Gaussian distribution exists is a question omitted usually. We will show a geometric solution for this question in Section 2. \par
When the expectation fixed,  M. Gelbrich \cite{ref6} proved the distance can be geometrized into a Riemannian metric on $SPD\left(n\right)$, called Wasserstein metric $g_w$, which can also be constructed via regarding $GL\left(n\right)$ as a fiber bundle on $SPD\left(n\right)$. In other words, Wasserstein distance is actually the geodesic distance based on this new metric constructed by the fiber bundle (or Riemannian submersion). This amazing result is proved by  \cite{ref3}, but  some analytic details have been omitted.  In this paper, from an original viewpoint , we reprove the facts and fill those essential blanks up. About the Wasserstein geometry, some geometric quantities have been calculated in some way \cite{ref3,ref4,ref8}, relying on conventional and complicated  geometric computing. Due to the fiber bundle structure \cite{ref20,ref19}, we present succinct expressions for geodesics, the exponential and the Riemannian connection. In the following, we give more geometric quantities, including Jacobi fields and curvatures. One of our main goals is to present this manifold more completely.

\subsection{Notation}
In the whole paper, we adopt conventional notations in algebra and geometry. Riemannian manifolds are denoted as pairs of ${\rm\left(manifold, metric\right)}$. For example, our main study object is $\left(SPD\left(n\right), g_w\right)$, meaning $SPD\left(n\right)$ endowed with Wasserstein metric. $\mathbb{R}^n$ is  the $n$-dimensional Euclidean space.$M\left(n\right)$ represents the set of $n\times n$ matrices and $Sym\left(n\right)$ means the set of $n\times n$ symmetric matrices. $T_{A}M$ is conventionally the tangent space of $M$ at a point ${A}$. $\Lambda$ always represents a diagonal $n\times n$ matrix. For an $n\times n$ matrix ${A}$, $\lambda\left({A}\right)$ or $\lambda_i\left({A}\right)$ means a  eigenvalue or the $i$-th eigenvalue of ${A}$, respectively. Then $x_i\sim \lambda_i$ is a pair of eigenvalues and one of its associated eigenvectors. The components of matrix ${A}$ will always be noted as $[{A}_{ij}]$. And identity matrix is denoted as $I$. In this paper, we tends to express points on manifolds by ${A},B$ while vector fields by ${X},{Y}$.

\subsection{Overview}
The paper is organized as follows. In the section 2, we introduce some basic knowledge of the Riemannian manifold $\left(SPD(n),g_W\right)$, including some definitions and a important Algorithm. We will consider the symmetry of the $\left(SPD(n),g_W\right)$ as well. In the section 3, we focus on geodesic. we prove the geodesic convexity and give the exponential and Wasserstein radius, which allows us to depict the manifold in some extent.  In the section 4, we express the Riemannian connection explicitly by comparing the total space and base space of the bundle structure. Section 5 is about Jacobi fields. In this part, we prove the non-existence of conjugate pair and cut locus, which shows the uniqueness of geodesic. Finally, in the section 6, we study curvatures. We obtain explicit expressions for the Riemannian curvature, sectional curvatures and the scalar curvature. Furthermore, we give the estimate of sectional curvatures, showing that the curvatures can be controlled by the minimal eigenvalue of points on $\left(SPD(n),g_W\right)$.

\section{Preliminary}
\subsection{Algebra about Sylvester Equation}

Sylvester equation is one of the most classical matrix equations. The solution of this equation plays a key role to understand the geometry of $\left(SPD\left(n\right),g_w\right)$. First of all, we shall list series of properties of this equation, which are utilized in the following arguments frequently. Then we recall a useful algorithm to calculate the explicit form of the solution. This algorithm paves a way for calculating and estimating the Wasserstein curvatures. Here we denote the solution of Sylvester equation with ${\Gamma}$, which means
${\Gamma} : M(n)\times SPD(n)\longrightarrow M(n)$, for any ${A} \in SPD(n),{X} \in M(n)$, and the matrix ${\Gamma}_{A}[{X}] \in M(n)$ satisfies
\begin{equation}\label{Sylv}
	{A} {\Gamma}_{A}[{X}]+{\Gamma}_{A}[{X}]  {A}={X}.
\end{equation}

Our discussion does not contain the well-posedness of Sylvester equation. From geometric aspects, we can ensure the existence and uniqueness of the solution in the particular case involved in this paper. More details about Sylvester equation are claimed in \cite{ref13}. Here we list some basic properties which will be used frequently in the following discussions. \par
\begin{proposition}\label{prop1}
	\ \par
	\begin{enumerate}
		\item $\Gamma_{A}\left({X}+k{Y}\right) = \Gamma_{A}[{X}]+k\Gamma_{A}[{Y}\,\ \forall k \in \mathbb{R} $.
		\item $\Gamma_{k{A}}[{X}\ = \frac{1}{k} \Gamma_{A}[{X}] ,\ \forall k \in \mathbb{R}$.
		\item $\Gamma_{{A}+B}[{X}] = \Gamma_{A}[{X}]-\Gamma_{A}\left[B \Gamma_{{A}+B}[{X}]+\Gamma_{{A}+B}[{X}] B\right].$
		\item $
		\Gamma_{A}[{A}{X}] =  {A}  \Gamma_{A}[{X}]\ ,\ \Gamma_{A}\left[{X}{A}\right] =  \Gamma_{A}[{X}]  {A}$.
		\item $\Gamma_{{A}^{-1}}[{X}] =  \Gamma_{A}[{A}{X}{A}]={A} \Gamma_{A}[{X}]  {A} $.
		\item $\Gamma_{Q{A} Q^{-1}}[Q{X} Q^{-1}] =  Q \Gamma_{A}[{X}]Q^{-1},\ \forall Q \in GL\left(n\right)$.
	\end{enumerate}
\end{proposition}

Proposition \ref{prop1} will determine the geometry on $\left(SPD\left(n\right),g_W\right)$ and be involved into every calculation repetitively and alternatively. We omit the proof since these properties are easy-checked.

Then we recall an algorithm to solve this kind of Sylvester equations, which offers an explicit expression of the solution. And this expression only depends on the eigenvalue decomposition. More details can be found in \cite{ref8}.

\begin{algorithm}[H]
\caption{radius outlier removal}
	\hspace*{0.02in} {\bf Input:}
	${A} \in SPD\left(n\right)$, ${X} \in Sym\left(n\right)$\\
	\hspace*{0.02in} {\bf Output:}
	$\Gamma_{A}[{X}]$
\begin{algorithmic}\label{algvest}
\State Eigenvalue decomposition, ${A} = Q\Lambda Q^{T}$, where $Q \in {O}\left(n\right), \Lambda:={\rm diag}[\lambda_1,\dots,\lambda_n]$
\State $C:= [C_{ij}]=Q^{T}{X}Q$
\State $E:=[E_{ij}]=\left[\frac{C_{ij}}{\lambda_i+\lambda_j}\right]$
\State \Return $\Gamma_{A}[{X}] = QEQ^{T}$
\end{algorithmic}
\end{algorithm}
The algorithm will be simplified if $A$ is specially a diagonal matrix $\Lambda$. This algorithm will be used frequently in the following passage, which helps us to comprehend the geometry of $SPD\left(n\right)$.


\subsection{Wasserstein Metric}
In this part, we will recall some concepts about Riemannian geometry and fiber bundles, and then introduce the Wasserstein metric on $SPD\left(n\right)$.
\begin{definition} For a principal bundle $\left(\widetilde{M},\widetilde{g}\right)$ with the structure group $G$ and projection $\sigma$ onto the based manifold $M$, any vector $W\in T_{\widetilde{A}}\widetilde{M}$ is {horizontal}, if and only if $\widetilde{g}\left(W,V\right)=0$, for all $V \in T_{\widetilde{A}}\widetilde{M}$ such that $d\sigma\left(V\right)=0$. We say $V\perp G[\widetilde{A}]$, where $G[\widetilde{A}]$ is the orbit of $\widetilde{A}$ under the group action. And if ${\rm d}\sigma \left(\widetilde{X}\right)={X} \in T_{A}M$, we call $\widetilde{X}$ as a {lift} of ${X}$, where $A = \sigma \left(\widetilde{A}\right)$.
\end{definition}

\begin{proposition}
	The general linear group with Euclidean metric $\left(GL\left(n\right),g_E\right)$ and $\sigma\left({\widetilde{A}}\right) = {\widetilde{A}}^T{\widetilde{A}}$, where $g_E\left(\widetilde{X},\widetilde{Y}\right):= {\rm tr}\left(\widetilde{X}^T\widetilde{Y}\right)$, for any $\widetilde{X},\widetilde{Y} \in T_{\widetilde{A}}GL(n)$, is a trivial principal bundles on $SPD\left(n\right)$, with orthogonal group ${O}\left(n\right)$ as the structure group.
\end{proposition}
This proposition establishes for two facts \cite{ref8}: $SPD\left(n\right)\cong\frac{GL\left(n\right)}{{O}\left(n\right)}$ and $g_E$ keeps invariant under the group action of ${O}\left(n\right)$.
\begin{proposition}
	For any ${\widetilde{A}} \in \left(GL\left(n\right),g_E\right)$, ${A} :=\sigma\left({\widetilde{A}}\right) = {\widetilde{A}}^T{\widetilde{A}}$, and any $ X \in T_{A}SPD\left(n\right)$, there is a unique $\widetilde{X}$ to be the horizontal lift of $X$ at $T_{\widetilde{A}}GL\left(n\right)$, satisfying
	\begin{equation} \label{HLevel}
		\widetilde{X} = {\widetilde{A}}\Gamma_{A}[X].
	\end{equation}
\end{proposition}
\begin{proof}
	By definitions, horizontal lift subjects to the following two conditions. From ${\rm d}\sigma \left(\widetilde{X}\right)=X$ and $\widetilde{X} \perp {O}\left(n\right)[{\widetilde{A}}]$, we get
	\begin{equation}\label{level1}
		\begin{cases}
			\widetilde{X}^{T}{\widetilde{A}}+{\widetilde{A}}^{T}\widetilde{X} = X \quad ,\\
			{\widetilde{A}}^{-1}\widetilde{X} =\widetilde{X}^{T}{\widetilde{A}}^{-T}.
		\end{cases}
	\end{equation}
	Thus, we have
	\begin{equation*}
		{\widetilde{A}}^T{\widetilde{A}}{\widetilde{A}}^{-1}\widetilde{X}+{\widetilde{A}}^{-1}\widetilde{X}{\widetilde{A}}^T{\widetilde{A}} = X,
	\end{equation*}
	which is equivalent to (\ref{HLevel}).
\end{proof}

\begin{definition}
	For any ${A}\in SPD\left(n\right)$, $X,Y\in T_{A}SPD\left(n\right)$, we define
	\begin{equation}\label{Gw metric}
		g_W|_{A}\left(X,Y\right)= {\rm tr}\left(\Gamma_{A}[Y]  {A} \Gamma_{A}[X]\right) = \frac{1}{2}{\rm tr}\left(\Gamma_{A}[Y]  X\right),
	\end{equation}
\end{definition}
where $g_W$ is a symmetric and non-degenerated bilinear tensor fields on $SPD\left(n\right)$. We call $g_W$ as Wasserstein metric.
\begin{proposition}
	The projection $\sigma:\left(GL\left(n\right),g_E\right)\rightarrow\left(SPD\left(n\right),g_W\right)$ is a Riemannian submersion {\rm \cite{ref7}}, which means ${\rm d}\sigma $ is surjective and
	\begin{equation}\label{submer}
		g_E\left(\widetilde{X},\widetilde{Y}\right)=g_W\left({\rm d}\sigma \left(\widetilde{X}\right),{\rm d}\sigma \left(\widetilde{Y}\right)\right) =g_W\left({X},{Y}\right).
	\end{equation}
\end{proposition}

\subsection{Symmetry}
Now we study the symmetry of $\left(SPD\left(n\right), g_W\right)$. We consider the invariability of Wasserstein metric under a special group action, and we claim that the orthogonal group ${O}\left(n\right)$ is isomorphic to a subgroup of the isometry group $ISO\left(SPD\left(n\right),g_W\right)$.\par
\begin{definition}
	The orthogonal action $\Psi:{O}\left(n\right)\times SPD\left(n\right) \to SPD\left(n\right)$ is defined by
	\begin{equation}
		\Psi_{O}\left({A}\right) =O{A}O^{T},\quad \forall O\in {O}\left(n\right), A\in SPD(n).
	\end{equation}
\end{definition}

\begin{theorem}\label{DCHEN}
	The orthogonal group ${O}\left(n\right)$ is isomorphic to a subgroup of the isometric group of $\left(SPD\left(n\right),g_W\right)$, as
	\begin{equation}\label{duichen}
		{\{\Psi_{O}\left({A}\right)\}}_{O\in {O}} \lhd ISO|{g_W}.
	\end{equation}
\end{theorem}

\begin{proof}
	First, we check that $\Psi$ is a group action of ${O}\left(n\right),\forall O_1,O_2\in {O}\left(n\right), A\in SPD(n)$,
	\begin{equation*}
		\begin{split}
			&\Psi_{I}{A}={A} ={\rm id}\left({A}\right),\\
			&\Psi_{O_1O_2^{-1}}{A}=O_1O_2^T {A}O_2O_1^T = \Psi_{O_1}\circ\Psi_{O_2^{-1}}{A}.
		\end{split}
	\end{equation*}
	
	Then we show that the tangent maps of these actions ${\rm d}\Psi_O$ are isometric,
	\begin{equation*}
		\begin{split}
			&\langle {\rm d}\Psi_O\left({X}\right),{\rm d}\Psi_O\left({Y}\right)\rangle|_{O{A}O^T} \\
			=\ & \frac{1}{2}{\rm tr}\left(\Gamma_{O{A}O^T} \left[O{X}O^T\right)O{Y}O^T\right]\\
			=\ & \frac{1}{2}{\rm tr}\left({Y}\Gamma_{A}[X]\right) ={\langle {X},{Y}\rangle}|_{A}.
		\end{split}
	\end{equation*}
	Thus, we prove the invariability of Wasserstein metric.
\end{proof} \par

We have such a fact that for an $n$-dimensional Riemannian manifold, the dimension of isometry group achieves maximum if and only if it has constant sectional curvature. Therefore, in section 5, we will show that $\left(SPD\left(n\right),g_W\right)$ has no constant sectional curvature, which means its symmetry degree less than the highest. The famous interval theorem \cite{ref1} about isometric group shows the nonexistence of isometric groups with dimension between $\frac{m\left(m-1\right)}{2}+1$ and $\frac{m\left(m+1\right)}{2}$, for any $m$-dimensional Riemannian manifold, where $m\neq 4$. On the other hand, (\ref{duichen}) shows that the dimension of Wasserstein isometric group is higher than the dimension of ${O}\left(n\right)$. Therefore, by ${\rm dim}\left(SPD\left(n\right)\right) = \frac{n^2+n}{2} \neq 4$ and ${\rm dim}\left({O}\left(n\right)\right) = \frac{n^2-n}{2}$, we have the following result.
\begin{corollary}
	$\left(SPD\left(n\right),g_W\right)$ has its symmetry degree controlled by
	\begin{equation}
		\frac{1}{2}\left(n-1\right)n\leq {\rm dim}\left(ISO|{g_W}\right) \leq \frac{1}{8}\left(n-1\right)n\left(n+1\right)\left(n+2\right)+1.
	\end{equation}
\end{corollary}
\par
According to Theorem \ref{DCHEN}, when we tend to study local geometric characteristics, we only need to consider the sorted diagonal matrices as the representational elements under the orthogonal action rather than all general points on $SPD\left(n\right)$. Therefore, some pointwise  quantities, such as the scalar curvature and the bounds of sectional curvatures, depend only on eigenvalues.\par

\section{Geodesic}
Although prior results  \cite{ref4} presented some expressions of geodesics and the Riemannian exponential, the well-posedness and uniqueness of geodesic equations have not been considered. Additionally, some questions, such as the domain of the exponential, extensions of geodesics, and the existence of the minimal geodesic jointing two arbitrary points, are waiting to be answered. In this section, we shall prove that $\left(SPD\left(n\right),g_W\right)$ is geodesic convex \cite{ref19}. Further, we calculate the maximal length for each geodesic, and then give the results about Wasserstein radius. Meanwhile, we will reprove some known results \cite{ref4}, including the expression of geodesics and the exponential, from an original viewpoint.
\subsection{Geodesic Convexity}
Wasserstein metric is not complete (we will see it later), which takes lots of trouble for researching the geometry of $\left(SPD\left(n\right),g_W\right)$. However, we tend to show that the whole Riemannian manifold $\left(SPD\left(n\right),g_W\right)$ is geodesic convex, which means that we can always find the minimal geodesic jointing any two points. To some extent, geodesic convexity may make up for the incompleteness.
\begin{definition}
	A curve $\widetilde{\gamma}\left(t\right)$ in the bundle  $\left(\widetilde{M},g_E\right)$ is said to be level, if and only if, $ \dot{\widetilde{\gamma}}\left(t\right)$ is horizontal in $T_{\widetilde{\gamma}\left(t\right)}{\widetilde{M}}$ for any $t$.
\end{definition}
\par
\begin{theorem}\label{thm31}
	For any $\ {A}_1,{A}_2 \in SPD\left(n\right)$, and choose $\widetilde{{A}_1}={{A}_1}^{\frac{1}{2}}$ as the fixed lift of ${A}_1$, there exists $\ \widetilde{{A}_2}\in GL\left(n\right)$ as a lift of ${A}_2$,
	\begin{equation}\label{LLF}
		\widetilde{{A}_2} = {{A}_1}^{-\frac{1}{2}}{\left({A}_1{A}_2\right)}^{\frac{1}{2}},
	\end{equation}
	such that the line segment $l\left(t\right)=t\widetilde{{A}_2}+\left(1-t\right){{A}_1}^{\frac{1}{2}}$
	, $t \in [0,1]$ is {level} and {non-degenerated},
\end{theorem}
Before its proof presented, we show that Theorem \ref{thm31} brings some geometrical and physical facts.\par
\begin{corollary}\label{geodesic convex}
	({geodesic convexity}) $\left(SPD\left(n\right),g_W\right)$ is a geodesic convex Riemannian manifold. Between any two points ${A}_1,\ {A}_2 \in SPD\left(n\right)$, there exists a minimal Wasserstein geodesic
	\begin{equation}\label{geodesic}
		\gamma\left(t\right) = {\left(1-t\right)}^2{A}_1+t\left(1-t\right)[{\left({A}_1{A}_2\right)}^{\frac{1}{2}}+{\left({A}_2{A}_1\right)}^{\frac{1}{2}}]+t^2{A}_2,\ t \in [0,1],
	\end{equation}
	where $\gamma\left(t\right)$ lies on $SPD\left(n\right)$ strictly.
\end{corollary}
\par
Let us recall the definition of Wasserstein distance (\ref{Wsst}). The lower energy bound does not ensure the existence of a real path to `earth-move' one distribution to another with the minimal price. However, by Theorem \ref{thm31} and Corollary \ref{geodesic convex}, we can actually ensure the existence of the path among Gaussian distributions with zero expectation.\par
To prove Theorem \ref{thm31}, we shall give two lemmas at first.\par
\begin{lemma}\label{lemma33}
	For any two matrices ${\widetilde{A}}_1,{\widetilde{A}}_2 \in GL\left(n\right)$, $l\left(t\right) = t{\widetilde{A}}_2+\left(1-t\right){\widetilde{A}}_1,\ t \in [0,1]$. $\dot{l}\left(t\right)$ is horizontal if and only if ${\rm det}\left({l}\left(t\right)\right) >0 $ and $\dot{l}\left(0\right)$ is horizontal at $T_{{\widetilde{A}}_1}GL\left(n\right)$.
\end{lemma}
\begin{proof}
	According to (\ref{level1}),
	$\dot{l}\left(0\right) = {\widetilde{A}}_2- {\widetilde{A}}_1$ is horizontal at $T_{{\widetilde{A}}_1}GL\left(n\right)$
	\begin{equation*}
		\begin{split}
			& \Leftrightarrow {\widetilde{A}}_1^{-1}\left({\widetilde{A}}_2- {\widetilde{A}}_1\right) \in Sym\left(n\right) \Leftrightarrow {\widetilde{A}}_1^{-1}{\widetilde{A}}_2 \in Sym\left(n\right)\\
			& \Leftrightarrow\noindent {\widetilde{A}}_2^{-1}{\widetilde{A}}_1 \in Sym\left(n\right)\Leftrightarrow {\widetilde{A}}_2^{-1}\left({\widetilde{A}}_2- {\widetilde{A}}_1\right) \in Sym\left(n\right).\\
		\end{split}
	\end{equation*}
	\ \ \ \ Thus, $\dot{l}\left(1\right)$ is also horizontal at $T_{{\widetilde{A}}_2}GL\left(n\right)$. For any $s$ such that $l\left(s\right) \in GL\left(n\right)$, we constrict the line segment by $l_s\left(t\right) := l\left(st\right)$, and the lemma can be induced by horizontal of $\dot{l}_s\left(1\right)$.
\end{proof} \par
Lemma \ref{lemma33} converts the levelness of a line segment into the horizontal of its initial vector, which brings quite convenience for the following proof.\par

\begin{lemma}\label{lemma34}
	For any ${A}_1,{A}_2 \in SPD\left(n\right)$, there exists $P ={{A}_1}^{-\frac{1}{2}}{\left({A}_1{A}_2\right)}^{\frac{1}{2}}{{A}_2}^{-\frac{1}{2}} \in {O}\left(n\right)$ such that $ \widetilde{\gamma}\left(t\right)=tP{{A}_2}^{\frac{1}{2}}+\left(1-t\right){{A}_1}^{\frac{1}{2}}$. When ${\rm det}\left(\widetilde{\gamma}\left(t\right)\right)>0$,  $\dot{\widetilde{\gamma}}\left(t\right)$ keeps horizontal at $T_{\widetilde{\gamma}\left(t\right)}GL\left(n\right)$.
\end{lemma}
\begin{proof}
	First, we have
	\begin{equation}\label{OP}
		\begin{split}
			P^TP & = {{A}_2}^{-\frac{1}{2}}{\left({A}_2{A}_1\right)}^{\frac{1}{2}}{{A}_1}^{-1}{\left({A}_1{A}_2\right)}^{\frac{1}{2}}{{A}_2}^{-\frac{1}{2}} \\
			& ={{A}_2}^{-\frac{1}{2}}{{A}_1}^{-1}{\left({A}_1{A}_2\right)}^{\frac{1}{2}}{\left({A}_1{A}_2\right)}^{\frac{1}{2}}{{A}_2}^{-\frac{1}{2}}\\
			&=I.
		\end{split}
	\end{equation}
	\ \ \ \ Then we only need to check that ${{A}_1}^{-\frac{1}{2}}\dot{\widetilde{\gamma}}\left(0\right) \in Sym\left(n\right)$. Due to Lemma \ref{lemma33}, we have
	\begin{equation}\label{lamme222}
		\begin{split}
			& {{A}_1}^{-\frac{1}{2}}\dot{\widetilde{\gamma}} = {{A}_1}^{-1}\left({\left({A}_1{A}_2\right)}^{\frac{1}{2}}-I\right) \in Sym\left(n\right)\\
			\Leftrightarrow\  &{{A}_1}^{-1}{\left({A}_1{A}_2\right)}^{\frac{1}{2}} \in Sym\left(n\right)\Leftrightarrow {{A}_1}^{-1}{\left({A}_1{A}_2\right)}^{\frac{1}{2}}={\left({A}_2{A}_1\right)}^{\frac{1}{2}}{{A}_1}^{-1},
		\end{split}
	\end{equation}
	which proves the lemma.
\end{proof} \par
\textbf{Remark}. In fact, finding the orthogonal matrix $P$ is equivalent to solving a Riccati equation. One can see \cite{ref4} for details.

To prove Theorem \ref{thm31}, the last step is to clarify the non-degeneration.
\begin{proof}
	(Theorem \ref{thm31}) Because of the symmetry from (\ref{duichen}), it suffices to prove the diagonal case ${A}_1 = \Lambda$. We shall prove ${\rm det}\left(\widetilde{\gamma}\left(t\right)\right)>0$, $\forall \ t \in [0,1]$. We have
	\begin{equation}\label{det}
		\begin{split}
			D\left(t\right)&:={\rm det}\left(\widetilde{\gamma}\left(t\right)\right)\\
			&= {\rm det}\left(t{\Lambda}^{-\frac{1}{2}}{\left(\Lambda {A}_2\right)}^{\frac{1}{2}}+\left(1-t\right){\Lambda}^{\frac{1}{2}}\right)  \\
			& = {\rm det}\left(t{\Lambda}^{-\frac{1}{2}}{\left(\Lambda {A}_2\right)}^{\frac{1}{2}}+\left(1-t\right){\Lambda}^{\frac{1}{2}}\right) \\
			& =\left(1-t\right)^n{\rm det}\left({\Lambda}^{\frac{1}{2}}\right){\rm det}\left(I+\frac{t}{\left(1-t\right)}{{\Lambda}^{-1}\left(\Lambda {A}_2\right)}^{\frac{1}{2}}\right).
		\end{split}
	\end{equation}
	
	Let $E\left(t\right):={\rm det}\left(I+\frac{t}{\left(1-t\right)}{{\Lambda}^{-1}\left(\Lambda {A}_2\right)}^{\frac{1}{2}}\right)$. Since $D\left(0\right),D\left(1\right)>0$ and $\lim\limits_{t\rightarrow1^-}\left(1-t\right)^n=0^+$, we see that $D\left(t\right)>0 \Longleftrightarrow E\left(t\right)>0,\ \forall t \in \left(0,1\right)$. Let $M:={\Lambda}^{-1}\left(\Lambda {A}_2\right)^{\frac{1}{2}}$.  By choosing  pairs $\{\bar{x}_i\sim\lambda_i\}$ of $\left(\Lambda {A}_2\right)^{\frac{1}{2}}$ as an orthonormal basis in $\mathbb{R}^n$, we have
	\begin{equation*}
		\lambda_i>0 \Longrightarrow \bar{x}_i^TM\bar{x}_i =\lambda_i \bar{x}_i^T{\Lambda}^{-1}\bar{x}_i>0.
	\end{equation*}
	Thus $M$ is positive definite.
	Then $E\left(t\right)={\rm det}\left(I+\frac{t}{\left(1-t\right)}M\right)>0$, for any $t\in [0,1]$. To sum up, the line segment $\widetilde{\gamma}\left(t\right)$ is  non-degenerated.
	
	\ \ \ \ Combining with Lemma \ref{lemma34}, the proof is done.
\end{proof} \par

\subsection{Exponential and Radius}
Following Lemma \ref{lemma33}, we can directly write down the Wasserstein logarithm on $SPD\left(n\right)$, for any ${A}_1\in SPD\left(n\right)$, ${\rm log}_{{A}_1}: SPD(n)\rightarrow T_{{A}_1}SPD(n)$
\begin{equation}\label{log}
	\log _{{A}_1}{{A}_2}={\rm d}\sigma |_{{{A}_1}^{\frac{1}{2}}}\dot{\widetilde{\gamma}}\left(0\right) = {\left({A}_1{A}_2\right)}^{\frac{1}{2}}+{\left({A}_2{A}_1\right)}^{\frac{1}{2}}-2{A}_1.
\end{equation}\par
By solving the inverse problem of above equation, we gain the expression of the {Wasserstein exponential}.\par
\begin{theorem}\label{explog}
	In a small open  ball ${B}\left(0,\varepsilon\right),\ \varepsilon>0$ in $T_{A} SPD\left(n\right) \cong {\mathbb{R}}^{\frac{1}{2}n\left(n+1\right)}$, the Wasserstein exponential at ${A}$,
	$\exp_{A}: \mathcal{B}\left(0,\varepsilon\right) \to SPD\left(n\right)$ is explicitly defined by
	\begin{equation}\label{exp}
		\exp_{A}{X} = {A}+{X}+\Gamma_{A}[{X}]  {A} \Gamma_{A}[{X}].
	\end{equation}
\end{theorem}
\begin{proof}
	By choosing the normal coordinates \cite{ref19} at ${A}$, there always exist neighborhoods where $\exp_{A}$ is well-defined. From (\ref{log}), given $\exp_{A}{X}$ well-defined, it satisfies
	\begin{equation}\label{exxxp}
		{\left({A} \exp_{A}{X}\right)}^{\frac{1}{2}}+{\left(\exp_{A}{X}  {A}\right)}^{\frac{1}{2}}= {X}+2{A}.
	\end{equation}
	This equation can convert to the Sylvester equation and we can express its solution as
	\begin{equation}\label{exxxp2}
		\begin{split}
			&{A}{A}^{-1}{\left({A} \exp_{A}{X}\right)}^{\frac{1}{2}}+{A}^{-1}{A}{\left(\exp_{A}{X}  {A}\right)}^{\frac{1}{2}} {A}^{-1}{A}= {X}+2{A} \\
			\Leftrightarrow\ & {A}[{A}^{-1}{\left({A} \exp_{A}{X}\right)}^{\frac{1}{2}}]+[{A}^{-1}{\left({A}\exp_{A}{X} \right)}^{\frac{1}{2}}]{A}= {X}+2{A} \\
			\Leftrightarrow\ & {A}\Gamma_{A}[{X}+2{A}] = {\left({A} \exp_{A}{X}\right)}^{\frac{1}{2}}.
		\end{split}
	\end{equation}
	Therefore, we have
	\begin{equation}\label{expp}
		\begin{split}
			\exp_{A}{X} &=\Gamma_{A}[{X}+2{A}]{A}\Gamma_{A}[{X}+2{A}] \\
			&=\left(\Gamma_{A}[{X}]+I\right){A}\left(\Gamma_{A}[{X}]+I\right)  \\
			& = {A}+{X}+\Gamma_{A}[{X}]  {A} \Gamma_{A}[{X}],
		\end{split}
	\end{equation}
	which finishes this proof.
\end{proof} \par
\textbf{Remark}. We call the first two terms ${A}+{X}$ as the Euclidean exponential, and the last term $\Gamma_{A}[{X}]  {A} \Gamma_{A}[{X}]$ as the Wasserstein correction for this bend manifold.
\begin{corollary}
	The geodesic equations with initial conditions $\gamma(0),\dot\gamma(0)$ on $\left(SPD\left(n\right),g_W\right)$ have the following explicit solution
	\begin{equation}\label{cdx2}
		\gamma\left(t\right) = \gamma\left(0\right)+t\dot{\gamma}\left(0\right)+t^2\Gamma_{\gamma\left(0\right)}[\dot{\gamma}\left(0\right)]  \gamma\left(0\right) \Gamma_{\gamma\left(0\right)}[\dot{\gamma}\left(0\right)], \ t \in \left(-\varepsilon,\varepsilon\right).
	\end{equation} \par
\end{corollary}

Subsequently, the next natural question is the maximal length of the extension of a geodesic. This question is equivalent to the largest domain of the exponential. We still focus on diagonal matrices.
\begin{theorem}\label{tofar}
	For any ${A} \in SPD\left(n\right)$ and ${X} \in T_{A}SPD\left(n\right)$, $\exp_{A}\left(t{X}\right):[0,\varepsilon)\rightarrow SPD\left(n\right)$ is well-defined if and only if
	\begin{equation}\label{far}
		\varepsilon_{max} =
		\begin{cases}
			-\frac{1}{\lambda_{min}}, & \mbox{if}\ \lambda_{min}<0,\\
			+\infty, & \mbox{if}\ \lambda_{min}\geq0,
		\end{cases}
	\end{equation}
	where $\lambda_{min}$ is the minimal eigenvalue of $\Gamma_{A}[{X}]$.
\end{theorem}

\begin{proof}
	Evidently, $\varepsilon_{max} = \min{\{s>0|\ {\rm det}\left(\exp_{A}\left(s{X}\right)=0\right)\}}$. By (\ref{expp}), we have
	\begin{equation}\label{rt}
		\begin{split}
			&{\rm det}\left(\exp_{A}\left(s{X}\right)\right)={\rm det}\left({A}\right){\rm det}^2\left(I+s\Gamma_{A}[{X}]\right)=0 \\
			\Leftrightarrow\  & {\rm det}\left(\Gamma_{A}[{X}]-\frac{-1}{s}I\right)=0\\ \Leftrightarrow\ &s = -\frac{1}{\lambda\left(\Gamma_{A}[{X}]\right)},
		\end{split}
	\end{equation}
	where $\lambda\left(\Gamma_{A}[{X}]\right)$ is the eigenvalue of $\Gamma_{A}[{X}]$. Thus $\varepsilon_{max} =  \min{\{-\frac{1}{\lambda\left(\Gamma_{A}[X]\right)}>0\}}.$
\end{proof}
\par
\noindent
\begin{corollary}\label{incomplete}
	Wasserstein metric $g_W$ on $SPD\left(n\right)$ is  incomplete.
\end{corollary}
\par
Corollary \ref{incomplete} can be directly obtained from Hopf-Rinow theorem  \cite{ref20}.
\par
Theorem \ref{tofar} and the next theorem help us to comprehend the size of $\left(SPD,g_W\right)$ from sense of each point.
\par
Figure \ref{figure1} and Figure \ref{figure2} show geodesics starting from different origins on $SPD\left(2\right)$. From this group of pictures, we can observe the outline of the manifold and some behaviors of geodesics.

\begin{figure}[htbp]
\centering
\begin{minipage}[t]{0.48\textwidth}
\centering
\includegraphics[width=6cm]{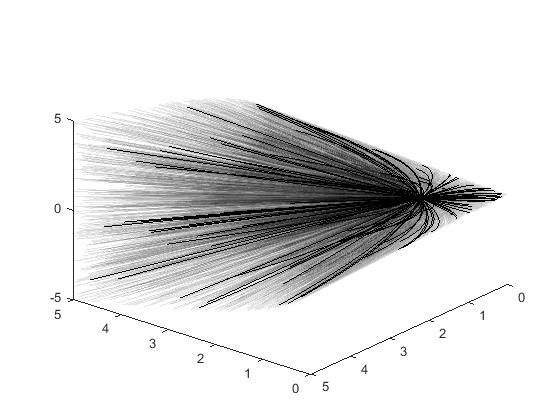}
\caption{geodesics starting from {\rm[1,0;0,1]}}\label{figure1}
\end{minipage}
\begin{minipage}[t]{0.48\textwidth}
\centering
\includegraphics[width=6cm]{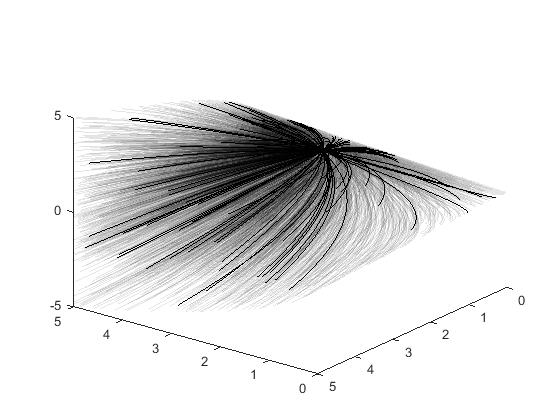}
\caption{geodesics starting from {\rm[3,1;1,1]}}\label{figure2}
\end{minipage}
\end{figure}

\begin{definition}
	For any ${A} \in SPD\left(n\right)$, we call $r\left({A}\right)$ the Wasserstein {radius} of $SPD\left(n\right)$ at ${A}$, if and only if $\exp_{A}$ is well-defined in all $\mathcal{B}\left(0,\varepsilon\right),\ 0\leq\varepsilon\leq r\left({A}\right)$.
\end{definition}\par
\noindent
\begin{theorem} The Wasserstein radius $r\left({A}\right): SPD\left(n\right)\rightarrow\left(0,+\infty\right)$ can be given by
	\begin{equation}\label{radius}
		r\left({A}\right) = \sqrt{\lambda_{min}\left({A}\right)/2},
	\end{equation}
	and the function $r({A})$ is continuous.
\end{theorem}
\begin{proof}
	Again, with (\ref{duichen}), we have $r\left({A}\right)=r\left(\Lambda\right)$ where ${A}=Q\Lambda Q^T$. Thus we still focus on $r\left(\Lambda\right)$. Following the definition and discussions around \eqref{far},
	\begin{equation}\label{rrrr}
		r\left(\Lambda\right):=\inf\{\varepsilon_{max}\left(V\right)|V\in T_\Lambda SPD\left(n\right),\ {\|V\|}_{g_W}=1\}.
	\end{equation}
	Equivalently, (\ref{rrrr}) has its dual expression. According to (\ref{rt}), we have
	\begin{equation}\label{rrrr2}
		r\left(\Lambda\right)=\inf\{{\|V\|}_{g_W}|V\in T_\Lambda SPD\left(n\right),\ {\rm det}\left(I+\Gamma_{\Lambda}[V]\right)\},
	\end{equation}
	\begin{equation}\label{rrrr3}
		\Leftrightarrow 2r^2\left(\Lambda\right) = \inf\{{{\rm tr}\left(V\Gamma_{\Lambda}[V]\right)}|\ \Gamma_{\Lambda}[V] \text{ has eigenvalue } \eta_k=-1\}.
	\end{equation}
	\ \ \ \ We tend to solve (\ref{rrrr3}). To avoid misconception, we denote $\{\lambda_i>0\}$ as eigenvalues of $\Lambda$, and ${\eta_j}$ as eigenvalues of $\Gamma_{\Lambda}[V]$. We choose the pairs $x_j\sim\eta_j, 1\leq j\leq n$ of $\Gamma_{\Lambda}[V]$ as an orthonormal basis of $\mathbb{R}^n$. Then we consider the inner product $g_W\left(V,V\right)$, and by calculation we get
	\begin{equation}\label{com1}
		\begin{split}
			{\rm tr}\left(V\Gamma_{\Lambda}[V]\right)=& \sum_{i=1}^{n}{x_i}^TV\Gamma_{\Lambda}[V]x_i \\
			=& \sum_{i=1}^{n}\eta_i{x_i}^TVx_i =\sum_{i=1}^{n}\eta_i{x_i}^T\left(\Lambda\Gamma_{\Lambda}[V]+\Gamma_{\Lambda}[V] \Lambda\right)x_i\\
			=&\sum_{i=1}^{n}\eta_i{\left(\Gamma_{\Lambda}[V] x_i\right)}^T\Lambda x_i+\eta_i{x_i}^TV\Gamma_{\Lambda}[V] x_i =\sum_{i=1}^{n}\eta_i^2{x_i}^T\Lambda x_i.
		\end{split}
	\end{equation}
	\ \ \ \ Because $\Lambda$ is positive and $\eta_k=-1$, ${\rm tr}\left(V\Gamma_{\Lambda}[V]\right)\geq \eta_k^2{x_k}^T\Lambda x_k$. The equality  ${\rm tr}\left(V\Gamma_{\Lambda}[V]\right) = {x_k}^T\Lambda x_k$ holds if and only if $\eta_j =0,\ \forall j\neq k$. In these cases, by Algorithm 1, we have
	\begin{equation}\label{tyr1}
		\left(\Gamma_{\Lambda}[V]\right)_{ij}= \frac{V_{ij}}{\lambda_i+\lambda_j}=
		\begin{cases}
			\eta_k=-1, & \mbox{if } i=j=k, \\
			0, & \mbox{otherwise}.
		\end{cases}
	\end{equation}
	We have
	\begin{equation}\label{Vij}
		V_{ij} = - 2\delta_{ik}\delta_{ik}\lambda_k,\ {\left(x_k\right)}_i=\delta_{ki},
	\end{equation}
	and thus
	\begin{equation}\label{tyr2}
		{\rm tr}\left(V\Gamma_{\Lambda}[V]\right) = {x_k}^T\Lambda x_k=\lambda_k.
	\end{equation}
	\ \ \ \ Especially, we obtain $\lambda_{min}$  as $\min\{{\rm tr}\left(V\Gamma_{\Lambda}[V]\right)\}$, when $\Lambda_k =\lambda_{min}$, and $V_{ij}=- 2\delta_{ik}\delta_{jk}\lambda_{min}$. The tangent vector $V$ is called the speed-degenerated direction. The function $\lambda_{min}\left({A}\right)$ is certainly continuous, hence $r({A})$ is also continuous.
\end{proof} \par
Due to the geodesic convexity, the radius actually defines the Wasserstein distance of a point on $SPD\left(n\right)$ to the 'boundary' of the manifold. It also measures the degenerated degree of a positive-definite symmetric matrix by $\sqrt{\lambda_{min}}$. 
Figure \ref{figure3} shows three maximal geodisical balls with different centers on $SPD\left(2\right)$. From the viewpoint of $\mathbb{R}^3$, the three balls have different sizes in the sense of Euclidean distance, but on $\left(SPD\left(2\right),g_W\right)$ all of them have the radius of $\frac{\sqrt{2}}{2}$.
\begin{figure}[H]
	\centering
	\includegraphics[scale=0.4]{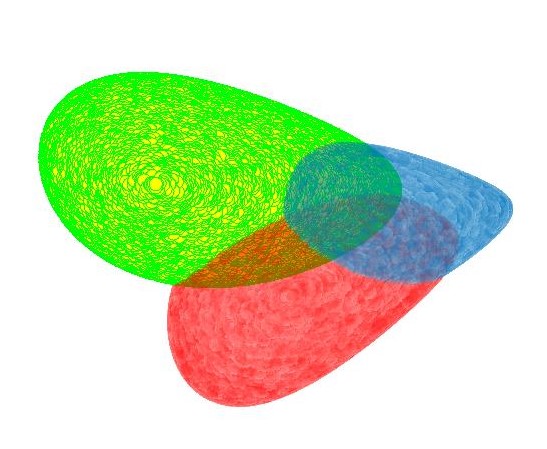}
	\vspace{0pt}
	\caption{three geodisical balls with same radius}\label{figure3}
\end{figure}

\section{Connection}
In this section, we will study the Riemannian connection of $\left(SPD\left(n\right),g_W\right)$, called {Wasserstein connection}. The flatness of $\left(GL\left(n\right),g_E\right)$ and the structure of the Riemannian submersion will take series of convenience to our work.\par
During computation, we denote both tensor actions of $g_W$ on $SPD\left(n\right)$ and $g_E$ on $GL\left(n\right)$ by  $\langle\cdot,\ \cdot\rangle$.  Then we denote the Euclidean connection as $D$, while the Wasserstein connection as $\nabla$.\par
The main idea to express the Wasserstein connection is to compute the level decomposition of the Euclidean covariant derivative of lifted vector fields. In a word, we shall prove:

\begin{theorem}\label{tm41}
	The Euclidean connection is a lift of the Wasserstein connection. For any smooth vector fields ${X}$ and ${Y}$ on $SPD\left(n\right)$, and $\widetilde{X}$ and $\widetilde{Y}$ are their level lifts, respectively, the following equation holds
	\begin{equation}
		{\rm d}\sigma | _{\widetilde{A}}\left(D_{\widetilde{X}}\widetilde{Y}\right) =\nabla_{X}{Y}, \forall \widetilde{\widetilde{A}} \in GL(n).
	\end{equation}
\end{theorem}

Before proving Theorem \ref{tm41}, we shall prove a key lemma which points the relation between the Lie-brackets on the total space and base space.\par
\begin{lemma}\label{lamm42} The level lift of vector fields commutes with Lie-brackets,
	\begin{equation}
		{\rm d}\sigma |_{\widetilde{A}}{[\widetilde{X},\widetilde{Y}]}=[{X},{Y}].
	,\forall \widetilde{\widetilde{A}} \in GL(n)\end{equation}
\end{lemma}
\begin{proof}(Lemma \ref{lamm42})
	On the flat manifold $\left(GL\left(n\right),g_E\right)$, the connection equals to the usual directional derivative in the Euclidean space. Therefore, for vector fields ${X}, {Y}$ on $SPD(n)$, we have
	\begin{equation}\label{upD}
		\begin{split}
			D_{\widetilde{X}}\widetilde{Y}&=\lim_{t\rightarrow0}\frac{1}{t}[\widetilde{Y}|_{{\widetilde{A}}+\widetilde{X}t}-\widetilde{Y}|_{\widetilde{A}}]   \\
			&=\lim_{t\rightarrow0}\frac{1}{t}\left(\left({\widetilde{A}}+\widetilde{X}t\right)\Gamma_{{A}+{X}t+o(t^2)}[{Y}|_{{A}+{X}t+o(t^2)}]-{\widetilde{A}}\Gamma_{A}[{Y}|_{A}]\right)
			\\
			&=\lim_{t\rightarrow0}\frac{\widetilde{A}}{t}\left(\Gamma_{{A}+{X}t}[{Y}|_{{A}+{X}t}]-\Gamma_{A}[{Y}|_{A}]\right)+\widetilde{X}\Gamma_{A}[{Y}]
			\\
			&=\lim_{t\rightarrow0}\frac{\widetilde{A}}{t}\left(\Gamma_{{A}+{X}t}[{Y}|_{{A}+{X}t}]-\Gamma_{A}[{Y}|_{A}]\right)+{\widetilde{A}}\Gamma_{A}[{X}]\Gamma_{A}[{Y}]
			\\
			&={\widetilde{A}}\left(\Gamma_{A}[{\rm d}{Y}\left({X}\right)]-\Gamma_{A}[{X} \Gamma_{A}[{Y}]+\Gamma_{A}[{Y}]  {X}]+\Gamma_{A}[{X}]\Gamma_{A}[{Y}]\right).
		\end{split}
	\end{equation}

	Putting $\Gamma_{A}[{X}]\Gamma_{A}[{Y}]-\Gamma_{A}[{Y}]\Gamma_{A}[{X}]$ into (\ref{Sylv}), we get
	
	\begin{equation}\label{eqn9}
		\begin{split}
			&\quad \Gamma_{A}[{X}\Gamma_{A}[{Y}]+\Gamma_{A}[{Y}]{X}-{Y}\Gamma_{A}[{X}]-\Gamma_{A}[{X}]{Y}]\\
			&= \left(\Gamma_{A}[{X}]\Gamma_{A}[{Y}]-\Gamma_{A}[{Y}]\Gamma_{A}[{X}]\right) -2\Gamma_{A}[\Gamma_{A}[{X}]\Gamma_{A}[{Y}]{A}-\Gamma_{A}[{Y}]\Gamma_{A}[{X}]{A}].
		\end{split}
	\end{equation} \par
	Then we compute the Lie-bracket
	\begin{equation} \label{eqn10}
		\begin{split}
			&\quad [\widetilde{X},\widetilde{Y}] := D_{\widetilde{X}}\widetilde{Y} - D_{\widetilde{Y}}\widetilde{X}\\
			&={\widetilde{A}}\Gamma_{A}[{\rm d}{Y}\left({X}\right)-{\rm d}{X}\left({Y}\right)]+{\widetilde{A}}\left(\Gamma_{A}[{X}]\Gamma_{A}[{Y}]-\Gamma_{A}[{Y}]\Gamma_{A}[{X}]\right)\\
			&\quad -{\widetilde{A}}\Gamma_{A}[{X}\Gamma_{A}[{Y}]+\Gamma_{A}[{Y}]{X}-{Y}\Gamma_{A}[{X}]-\Gamma_{A}[{X}]{Y}]\\
			&=\widetilde{[{X},{Y}]}+2{\widetilde{A}}\Gamma_{A}[\Gamma_{A}[{X}]\Gamma_{A}[{Y}]{A}-\Gamma_{A}[{Y}]\Gamma_{A}[{X}]{A}],
		\end{split}
	\end{equation}
	where the last equality in (\ref{eqn10}) comes from (\ref{eqn9}).\par
	Finally we show the second term in (\ref{eqn10}) is vertical. In fact, we have
	\begin{equation} \label{Snotlevel}
		\begin{split}
			&\quad {\rm d}\sigma |_{\widetilde{A}}\left({\widetilde{A}}\Gamma_{A}[\Gamma_{A}[{X}]\Gamma_{A}[{Y}]{A}-\Gamma_{A}[{Y}]\Gamma_{A}[{X}]{A}]\right)\\
			&={A}\Gamma_{A}[\Gamma_{A}[{X}]\Gamma_{A}[{Y}]{A}-\Gamma_{A}[{Y}]\Gamma_{A}[{X}]{A}]+\Gamma_{A}[{A}\Gamma_{A}[{Y}]\Gamma_{A}[{X}]-{A}\Gamma_{A}[{X}]\Gamma_{A}[{Y}]]{A}\\
			&={A}\Gamma_{A}[\Gamma_{A}[{X}]\Gamma_{A}[{Y}]-\Gamma_{A}[{Y}]\Gamma_{A}[{X}]]{A}+{A}\Gamma_{A}\left(\Gamma_{A}{Y}\Gamma_{A}{X}-\Gamma_{A}{X}\Gamma_{A}{Y}\right){A}\\
			&= \Gamma_{{A}^{-1}}[\Gamma_{A}[{X}]\Gamma_{A}[{Y}]{A}-\Gamma_{A}[{Y}]\Gamma_{A}[{X}]{A}+\Gamma_{A}[{Y}]\Gamma_{A}[{X}]{A}-\Gamma_{A}[{X}]\Gamma_{A}[{Y}]{A}] = 0.
		\end{split}
	\end{equation}
	Thus the proof for Lemma \ref{lamm42} has been done.
\end{proof} \par
By Lemma \ref{lamm42}, the proof for
Theorem \ref{tm41} is clarified.
\begin{proof}(Theorem \ref{tm41})
	For any smooth vector field ${Z}$ on $SPD\left(n\right)$, and its level lift $\widetilde{{Z}}$, we have
	\begin{equation*}
		\begin{split}
			&\quad \langle\nabla_{X}{Y},{Z}\rangle=\langle\widetilde{\nabla_{X}{Y}},\widetilde{{Z}}\rangle \\
			&= \frac{1}{2}\left({X}\langle {Y},{Z}\rangle+{Y}\langle {Z},{X}\rangle-{Z}\langle {X},{Y}\rangle+\langle {Z},[{X},{Y}]\rangle+\langle {Y},[{Z},{X}]\rangle-\langle {X},[{Y},{Z}]\rangle\right)\\
			&= \frac{1}{2}\left(\widetilde{X}\langle\widetilde{Y},\widetilde{{Z}}\rangle+\widetilde{Y}\langle \widetilde{{Z}},\widetilde{X}\rangle-\widetilde{{Z}}\langle\widetilde{X},\widetilde{Y}\rangle+\langle \widetilde{{Z}},\widetilde{[{X},{Y}]}\rangle+\langle\widetilde{Y},\widetilde{[{Z},{X}]}\rangle-\langle \widetilde{X},\widetilde{[{Y},{Z}]}\rangle\right)\\
			&= \frac{1}{2}\left(\widetilde{X}\langle\widetilde{Y},\widetilde{{Z}}\rangle+\widetilde{Y}\langle \widetilde{{Z}},\widetilde{X}\rangle-\widetilde{{Z}}\langle\widetilde{X},\widetilde{Y}\rangle+\langle \widetilde{{Z}},[\widetilde{X},\widetilde{Y}]\rangle+\langle\widetilde{Y},[\widetilde{{Z}},\widetilde{X}]\rangle-\langle \widetilde{X},[\widetilde{Y},\widetilde{{Z}}]\rangle\right)\\
			&=\langle{D_{\widetilde{X}}\widetilde{Y}},\widetilde{{Z}}\rangle.
		\end{split}
	\end{equation*}\par
	By the arbitrariness of ${Z}$, theorem \ref{tm41} is proved. \par
 This proof implies that ${\rm d}\sigma |_{\widetilde{A}}\left(D_{\widetilde{X}}\widetilde{Y}\right)$ is independent on ${\widetilde{A}}$ chosen among a fixed fiber.
\end{proof} \par
Theorem \ref{tm41} has a direct corollary which is one of essential results in this paper.
\begin{corollary}
	Wasserstein connection has an explicit expression:
	\begin{equation}\label{connection}
		\nabla_{X}{Y} = {\rm d}{Y}\left({X}\right)-\Gamma_{A}[{X}]{A}\Gamma_{A}[{Y}]-\Gamma_{A}[{Y}]{A}\Gamma_{A}[{X}].
	\end{equation}
\end{corollary}
\begin{proof} From Theorem \ref{tm41} and (\ref{eqn10}), we have
	\begin{equation*}
		\begin{split}
			\nabla_{X}{Y} &= {\rm d}\sigma |_{\widetilde{A}}\left(D_{\widetilde{X}}\widetilde{Y}\right)\\
			&={\widetilde{A}}^{T}D_{\widetilde{X}}\widetilde{Y}+D^{T}_{\widetilde{X}}\widetilde{Y}{\widetilde{A}}\\
			&={\rm d}{Y}\left({X}\right)-\left({X} \Gamma_{A}[{Y}]+\Gamma_{A}[{Y}]  {X}\right)+ {A}\Gamma_{A}[{X}]\Gamma_{A}[{Y}]+\Gamma_{A}[{Y}]\Gamma_{A}[{X}]{A}\\
			&={\rm d}{Y}\left({X}\right)-\Gamma_{A}[{X}]{A}\Gamma_{A}[{Y}]-\Gamma_{A}[{Y}]{A}\Gamma_{A}[{X}].
		\end{split}
	\end{equation*}
	The linearity, Leibnitz's law and symmetry of Wasserstein connection are easy-checked from the expression.
\end{proof}

\begin{definition}
	The horizontal component of lifted covariant derivative of ${Y}$ along ${X}$ is a vector field in $GL\left(n\right)$ whose value at ${\widetilde{A}}$ is defined by
	\begin{equation}
		\mathcal{T}_{\widetilde{A}}\left({X},{Y}\right) :=D_{\widetilde{X}}\widetilde{Y}-\widetilde{\nabla_{X}{Y}}.
	\end{equation}
	The whole vector field is denoted as $\mathcal{T}\left({X},{Y}\right)$.
\end{definition}
\begin{theorem}
	$\mathcal{T}_{\widetilde{A}}\left( \cdot,\cdot\right)$ is a antisymmetric bilinear map: $T_{A}SPD\left(n\right)\otimes T_{A}SPD\left(n\right)\to T_{\widetilde{A}}GL\left(n\right)$, and it satisfies
	\begin{equation}
		\mathcal{T}_{\widetilde{A}}\left({X},{Y}\right)={\widetilde{A}}\Gamma_{A}[\Gamma_{A}[{X}]\Gamma_{A}[{Y}]-\Gamma_{A}[{Y}]\Gamma_{A}[{X}]]{A}.
	\end{equation}
\end{theorem}

\begin{proof} For any $ V \in T_{\widetilde{A}}GL\left(n\right)$, $V$ has the orthogonal decomposition as
	\begin{equation*}
		V=V_{T}+V_{H}:={\widetilde{A}}^{-T}F+{\widetilde{A}}S,
	\end{equation*}
	where $S$ is a symmetric matrix, and $F$ is antisymmetric. Hence,
	\begin{equation}\label{Vv}
		V_{T} =  {\widetilde{A}}^{-T}\Gamma_{{A}^{-1}}[{\widetilde{A}}^{-1}V - V^{T}{\widetilde{A}}^{-T}].
	\end{equation}
	Combining (\ref{upD}) with (\ref{Vv}), we have
	\begin{equation}\label{SV}
		\begin{split}
			\mathcal{T}_{\widetilde{A}}\left({X},{Y}\right) &= {\widetilde{A}}^{-T}\Gamma_{{A}^{-1}}[\Gamma_{A}[{X}]\Gamma_{A}[{Y}]-\Gamma_{A}[{Y}]\Gamma_{A}[{X}]] \\&={\widetilde{A}}\Gamma_{A}[\Gamma_{A}[{X}]\Gamma_{A}[{Y}]-\Gamma_{A}[{Y}]\Gamma_{A}[{X}]]{A}.
		\end{split}
	\end{equation}
	(\ref{SV}) shows that $\mathcal{T}_{\widetilde{A}}\left({X},{Y}\right)$ depends only on ${\widetilde{A}}$ and the vectors on $T_{A}SPD\left(n\right)$. Meanwhile the multi-linearity and $\mathcal{T}_{\widetilde{A}}\left({X},{Y}\right)=-\mathcal{T}_{\widetilde{A}}\left({Y},{X}\right)$ are easy-checked.
\end{proof} \par
Recalling (\ref{eqn10}), we can also find that
\begin{equation*}
	[\widetilde{X},\widetilde{Y}] =\widetilde{[{X},{Y}]}+2\mathcal{T}\left({X},{Y}\right).
\end{equation*}\par
In the following parts, we will show the tensor $\mathcal{T}\left({X},{Y}\right)$ takes a  fundamental role for computing curvature.
\section{Jacobi Field}
As the Wasserstein exponential is given in section 2.2 with the explicit form,  we can clearly understand the behaviours of the geodesics. In this part, we will study the Wasserstein Jacobi fields on $SPD\left(n\right)$.\par
Come straight to the point. Jacobi fields can be directly constructed by the geodesic variation via the exponential.\par
\begin{theorem}\label{tm51}
	Along a geodesic $\gamma\left(t\right)$ with $\gamma\left(0\right)={A}\in SPD\left(n\right),\ \dot{\gamma}\left(0\right)={X}\in T_{A}SPD\left(n\right)$, there exists a unique normal Jacobi vector field $J\left(t\right)$ with initial conditions $J\left(0\right)=0, \nabla_{\dot{\gamma}\left(0\right)}J\left(t\right)={Y}\in T_{A}SPD\left(n\right)$, where $\langle {X},{Y}\rangle|_{A}=0$. We have
	\begin{equation}\label{Jacobi}
		J\left(t\right) = t{Y}+t^2\left(\Gamma_{A}[{X}]{A}\Gamma_{A}[{Y}]+\Gamma_{A}[{Y}]{A}\Gamma_{A}[{X}]\right).
	\end{equation}
\end{theorem}
As in \cite{ref7} $J\left(t\right)$ is constructed by
\begin{equation}\label{geodisical variational}
	J\left(t\right):= \left.\dfrac{\partial}{\partial s}\right|_{s=0}\exp_{A}t\left({X}+s{Y}\right).
\end{equation}\par
Substituting \eqref{exp} into \eqref{geodisical variational}, Theorem \ref{tm51} comes from direct computation.
\begin{theorem}
	There exists {no conjugate pair} on $\left(SPD\left(n\right),g_W\right)$.
\end{theorem}
\begin{proof} We prove by contradiction. Suppose that there exists a Jacobi field $J\left(t\right)\not\equiv0$, lying on the geodesic $\gamma\left(t\right)$, where $\gamma\left(0\right)={A}$, $J\left(0\right)=0$, and $p>0$ such that $J\left(p\right)=0$. We still denote $\nabla_{\dot{\gamma}\left(0\right)}J\left(t\right):={Y}$, $\dot{\gamma}\left(0\right):={X}$. Thus, we have
	\begin{equation*}
		J\left(p\right)=p{Y}+p^2\left(\Gamma_{A}[{X}]{A}\Gamma_{A}[{Y}]+\Gamma_{A}[{Y}]{A}\Gamma_{A}[{X}]\right)=0,
	\end{equation*}
	and thus we have
	\begin{equation*}
		\begin{split}
			& -p\left(\Gamma_{A}[{X}]{A}\Gamma_{A}[{Y}]+\Gamma_{A}[{Y}]{A}\Gamma_{A}[{X}]\right)={Y}={A}\Gamma_{A}[{Y}]+\Gamma_{A}[{Y}]  {A} \\
			\Leftrightarrow & \Gamma_{A}[{Y}]  {A}\left(I+p\Gamma_{A}[{X}]\right) +\left(I+p\Gamma_{A}[{X}]\right){A}\Gamma_{A}[{Y}]=0.
		\end{split}
	\end{equation*}
	\ \ \ \ According to discussions around (\ref{rt}), we know that $I+p\Gamma_{A}[{X}]$ is positive-definite as long as ${\gamma}\left(p\right)$ is well-defined. Due to ${Y}\neq 0$, and $\Gamma_{A}[{Y}]\neq0$,  we assume $\lambda \neq 0$ is an eigenvalue of $\Gamma_{A}[{Y}]$, and $x$ is an associated  eigenvector. Subsequently,
	\begin{equation}\label{fzf}
		\begin{split}
			& x^T\Gamma_{A}[{Y}] {A}\left(I+p\Gamma_{A}[{X}]\right)x +x^T\left(I+p\Gamma_{A}[{X}]\right){A}\Gamma_{A}[{Y}]x\\
			=& \lambda x^T  {A}\left(I+p\Gamma_{A}[{X}]\right)x +\lambda x^T\left(I+p\Gamma_{A}[{X}]\right){A}x \\
			=& 2\lambda x^T  {A}\left(I+p\Gamma_{A}[{X}]\right)x =0,
		\end{split}
	\end{equation}
	which contradicts to $\|x\|=1$ and the positive definiteness of ${A}\left(I+p\Gamma_{A}[{X}]\right)$.
\end{proof}\par
\begin{theorem}
	For any tow points in $(SPD(n),g_W)$, there exists a unique geodesic jointing them, i.e. there is no cut locus on any geodesic.
\end{theorem}

\begin{proof}
	From \cite{ref8} we see that the shortest distance among two fibers in bundle equals to the length of the geodesics, which informs the geodesics defined as Corollary \ref{geodesic convex} are always the shortest. For any ${A}_1, {A}_2 \in SPD\left(n\right)$, $P_1, P_2 \in {O}\left(n\right)$, the distance from the fiber $\sigma^{-1}
	\left({A}_1\right)$ to the fiber $\sigma^{-1}
	\left({A}_1\right)$ is defined as
	\begin{equation*}
		\inf_{P_1, P_2\in{O}(n)} \left\|P_1{{A}_1}^{\frac{1}{2}}-P_2{{A}_2}^{\frac{1}{2}}\right\|.
	\end{equation*}
	\ \ \ \ From Lemma \ref{lemma33} and Lemma \ref{lemma34}, due to the compactness of the  structure group ${O}(n)$, the distance will be achieved by the length of a level and non-degenerated line segment. Therefore, we have
	\begin{equation*}
		\begin{split}
			& \inf_{P_1, P_2\in{O}(n)}\left\|P_1{{A}_1}^{\frac{1}{2}}-P_2{{A}_2}^{\frac{1}{2}}\right\| \\
			= & \min_{P_1,P_2}\left({\rm tr}\left(P_2{{A}_2}^{\frac{1}{2}}-P_1{{A}_1}^{\frac{1}{2}}\right)^T\left(P_2{{A}_2}^{\frac{1}{2}}-P_1{{A}_1}^{\frac{1}{2}}\right)  \right)^{\frac{1}{2}}\\
			= & \min_{P_1,P_2}\left({\rm tr}\left({A}_1\right)+{\rm tr}\left({A}_2\right)-2{\rm tr}\left(P_2^TP_1{{A}_1}^{\frac{1}{2}}{{A}_2}^{\frac{1}{2}}\right)\right)^{\frac{1}{2}}\\
			= & \left({\rm tr}\left({A}_1\right)+{\rm tr}\left({A}_2\right)-2\max_{P_1,P_2} {\rm tr}\left(P_2^TP_1{{A}_1}^{\frac{1}{2}}{{A}_2}^{\frac{1}{2}}\right)\right)^{\frac{1}{2}}.
		\end{split}
	\end{equation*}
	
	By the results of matrix analysis \cite{ref12}, we see that for $P_1,P_2 \in {O}(n)$, ${\rm tr}\left(P_1{{A}_1}^{\frac{1}{2}}{{A}_2}^{\frac{1}{2}}P_2\right)$ achieves the maximum if and only if $P_1{{A}_1}^{\frac{1}{2}}{{A}_2}^{\frac{1}{2}}P_2$ is orthogonal diagonalizable. In fact, $P$ in Lemma \ref{lemma34} maximizes ${\rm tr}\left(P{{A}_1}^{\frac{1}{2}}{{A}_2}^{\frac{1}{2}}\right)$ , where $P ={{A}_1}^{-\frac{1}{2}}{\left({A}_1{A}_2\right)}^{\frac{1}{2}}{{A}_2}^{-\frac{1}{2}} $ and $\max\limits_{P}\left({\rm tr}\left(P{{A}_1}^{\frac{1}{2}}{{A}_2}^{\frac{1}{2}}\right)\right) = {\rm tr}\left({A}_1{A}_2\right)^{\frac{1}{2}}$. Thus the minimal geodesic distance of $\left(SPD\left(n\right),g_W\right)$ exactly equals to the Wasserstein distance.
	
	Then a basic geometric observation \cite{ref7} claims that if a shortest geodesic can be extended, it must be unique. Therefore, the uniqueness of the geodesic comes from the minimality proved above.
\end{proof}\par

{\bf Remark}. The non-existence of the conjugate pair actually implies the non-existence of the cut locus in our cases.

Up to the present, we have proved both the {existence} and {uniqueness} of  Wasserstein geodesic.\par
As well known that the behaviours of Jacobi fields are controlled by curvatures. Since Jacobi fields on $\left(SPD\left(n\right),g_W\right)$ is completely known, the Jacobi equation actually provides a method to calculate curvatures. Unfortunately, the computating is extremely complicated. We have to find another way for curvatures. One knows that the non-negative curvature always means geodesics convergent and brings conjugate pairs. But $\left(SPD\left(n\right),g_W\right)$ is weird  since it admits non-negative curvatures but without any conjugate pair.

\section{Curvature}
Our ultimate aim is to understand the Riemannian curvature of $\left(SPD\left(n\right),g_W\right)$. We denote the Euclidean curvature on bundle (null entirely) as $\widetilde{R}$, and the Wasserstein (Riemannian) curvature on $\left(SPD\left(n\right),g_W\right)$ as $R$.\par
\subsection{Riemannian Curvature Tensor}

\begin{theorem}\label{tm61}
	For any ${A} \in SPD\left(n\right)$, and ${X},{Y}$ are smooth vector fields on $SPD\left(n\right)$ , the Wasserstein  curvature tensor $R\left({X},{Y},{X},{Y}\right):=\langle R_{{X}{Y}}{X},{Y}\rangle _{A}$ at ${A}$ has an explicit expression
	\begin{equation}
		R\left({X},{Y},{X},{Y}\right) = 3{\rm tr}\left(\Gamma_{A}[{X}]{A}\Gamma_{A}[\Gamma_{A}[{X}]\Gamma_{A}[{Y}]-\Gamma_{A}[{Y}]\Gamma_{A}[{X}]]{A}\Gamma_{A}[{Y}]\right).
	\end{equation}
\end{theorem}

\begin{proof}(Theorem \ref{tm61})
	We attempt to compare  $R|_{A}\left({X},{Y},{X},{Y}\right)$ with $\widetilde{R}|_{\widetilde{A}}\left(\widetilde{X},\widetilde{Y},\widetilde{X},\widetilde{Y}\right)$ as follows
	\begin{equation}\label{eqn23}
		\begin{split}
			\widetilde{R}\left(\widetilde{X},\widetilde{Y},\widetilde{X},\widetilde{Y}\right):&=\langle \widetilde{R}_{\widetilde{X}\widetilde{Y}}\widetilde{X},\widetilde{Y} \rangle \\
			&=D_{[\widetilde{X},\widetilde{Y}]}\widetilde{X},\widetilde{Y}\rangle-\langle D_{\widetilde{X}}D_{\widetilde{Y}}\widetilde{X},\widetilde{Y}\rangle+\langle D_{\widetilde{X}}D_{\widetilde{Y}}\widetilde{X},\widetilde{Y}\rangle \\
			&=\langle D_{\widetilde{[{X},{Y}}]}\widetilde{X},\widetilde{Y}\rangle+\langle D_{2\mathcal{T}\left({X},{Y}\right)}\widetilde{X},\widetilde{Y}\rangle - \langle D_{\widetilde{X}}\widetilde{\nabla_{Y}{X}},\widetilde{Y}\rangle\\
			&\quad - \langle D_{\widetilde{X}}\mathcal{T}\left({Y},{X}\right),\widetilde{Y}\rangle+ \langle D_{\widetilde{Y}}\widetilde{\nabla_{X}{X}},\widetilde{Y}\rangle+ \langle D_{\widetilde{Y}}\mathcal{T}\left({X},{X}\right),\widetilde{Y}\rangle\\
			&=\langle \nabla_{[{X},{Y}]}{X},{Y}\rangle+2\langle D_{\mathcal{T}\left({X},{Y}\right)}\widetilde{X},\widetilde{Y}\rangle
			- \langle \nabla_{X}{\nabla_{Y}{X}},{Y}\rangle\\
			&\quad+\langle \mathcal{T}\left({Y},{X}\right),\mathcal{T}\left({X},{Y}\right)\rangle
			+ \langle \nabla_{Y}{\nabla_{X}{X}},{Y}\rangle-\langle \mathcal{T}\left({X},{X}\right),\mathcal{T}\left({Y},{Y}\right)\rangle.\\
		\end{split}
	\end{equation}
	By the definition of the usual derivative, we have
	\begin{equation}\label{eqn24}
		D_{\mathcal{T}\left({X},{Y}\right)}\widetilde{X} = \lim_{t\rightarrow0}\frac{1}{t}\left(\left({\widetilde{A}}+\mathcal{T}\left({X},{Y}\right)t\right)\Gamma_{A}[{X}]-{\widetilde{A}}\Gamma_{A}[{X}]\right)=\mathcal{T}\left({X},{Y}\right)\Gamma_{A}[{X}].
	\end{equation}
	By Lemma \ref{lamm42} and (\ref{Snotlevel}), we have $\langle [\mathcal{T}\left({X},{Y}\right),\widetilde{X}],\widetilde{Y}\rangle=\langle \mathcal{T}\left({X},{Y}\right),\widetilde{Y}\rangle=0$, and then we have
	\begin{equation}\label{eqn25}
		\langle D_{2\mathcal{T}\left({X},{Y}\right)}\widetilde{X},\widetilde{Y}\rangle = -2\langle \mathcal{T}\left({X},{Y}\right),D_{\widetilde{X}}\widetilde{Y}\rangle = -2\langle \mathcal{T}\left({X},{Y}\right),{\mathcal{T}\left({X},{Y}\right)}\rangle.
	\end{equation}\par
	Therefore, by putting (\ref{eqn24}) and (\ref{eqn25}) into (\ref{eqn23}) and the anti-symmetry of $\mathcal{T}\left(\cdot,\cdot\right)$, we get
	\begin{equation}\label{Re and Rw}
		\begin{split}
			\widetilde{R}\left(\widetilde{X},\widetilde{Y},\widetilde{X},\widetilde{Y}\right)&= R\left({X},{Y},{X},{Y}\right)+3\langle \mathcal{T}\left({X},{Y}\right)\Gamma_{A}[{X}],{\widetilde{A}}\Gamma_{A}[{Y}]\rangle\\
			&= R\left({X},{Y},{X},{Y}\right)-3\langle \mathcal{T}\left({X},{Y}\right),\mathcal{T}\left({X},{Y}\right)\rangle.\\
		\end{split}
	\end{equation}
	Thus, by the first equality from \eqref{Re and Rw} and $\widetilde{R}\equiv0$, we obtain the explicit expression for the Wasserstein curvature
	\begin{equation}\label{Rw expression}
		\begin{split}
			R\left({X},{Y},{X},{Y}\right) &= - 3\langle \mathcal{T}\left({X},{Y}\right)\Gamma_{A}[{X}],{\widetilde{A}}\Gamma_{A}[{Y}]\rangle \\&=  3{\rm tr}\left(\Gamma_{A}[{X}]{A}\Gamma_{A}[\Gamma_{A}[{X}]\Gamma_{A}[{Y}]-\Gamma_{A}[{Y}]\Gamma_{A}[{X}]]{A}\Gamma_{A}[{Y}]\right).\
		\end{split}
	\end{equation}
\end{proof}

Noticing ${\| \mathcal{T}\left({X},{Y}\right)\|}^2\geq0$, from above theorem we can obtain the following corollary.
\begin{corollary}\label{nonnegtive}
	$\left(SPD\left(n\right),g_W\right)$ has non-negative curvatures, namely
	\begin{equation}\label{posit}
		R\left({X},{Y},{X},{Y}\right)\geq0.
	\end{equation}
\end{corollary}
By solving the Sylvester equation with Algorithm 1, we can simplify the expression. We give the sectional curvature $K$ of the section ${\rm span}\{{X}\left({A}\right),{Y}\left({A}\right)\}$
\begin{equation}\label{Kw expression}
	\begin{split}
		K|_{A}\left({X},{Y}\right)&= \frac{R\left({X},{Y},{X},{Y}\right)}{\langle {X},{X}\rangle \langle {Y},{Y}\rangle-{\langle {X},{Y}\rangle}^2}\\
		&=12\frac{{\rm tr}\left(E_{X} \Lambda\Gamma_{\Lambda}{[E_{X},E_{Y}]} \Lambda E_{Y}\right)}{{\rm tr}\left(E_{X}C_{X}\right){\rm tr}\left(E_{Y}C_{Y}\right)-{\rm tr}^2\left(E_{X}C_{Y}\right)},
	\end{split}
\end{equation}
where we use the same donations as Algorithm 1.
We can also obverse that the sectional curvature conforms to the inverse ratio law
\begin{equation}\label{inverse ratio law}
	K|_{k\Lambda}\left({X},{Y}\right) = \frac{1}{k}K|_{\Lambda}\left({X},{Y}\right),\ \forall k \in \mathbb{R}.
\end{equation}
These results conform with our visualized views of $\left(SPD\left(n\right),g_W\right)$ presented in Figure \ref{figure1} and Figure \ref{figure2}, where the manifold tends to be flat when $k$ increases.

\subsection{Estimate}
Conventionally, we consider each fixed diagonal matrix $\Lambda$  on $\left(SPD\left(n\right),g_W\right)$. Define $\{S^{p,q}\}$ as
\begin{equation}
	S^{p,q} = [S^{p,q}_{ij}], \quad S^{p,q}_{ij} = \delta^{p}_i\delta^{q}_j+\delta^{q}_i\delta^{p}_j,
\end{equation}
where the superscripts $p,q$ mark the nonzero elements and $\delta$ is the Kronecker delta. In fact, $\{S^{p,q}| 1\leq p\leq q\leq n\}$ is a basis naturally.
For simplicity, we sometimes sign $S^{p,q}$, $S^{r,t}$ with $S_1$ and $S_2$.
By this way, we can express the curvature under this basis.\par
Before that, we introduce a useful symbol, which indicates how much symmetric pairs in the  Cartesian product of two script-sets.
\begin{definition}
	For any two finite script-sets  $\eta_1= \{\alpha_i|\alpha_i \in \mathbb{N}\}$, $\eta_2=\{{\beta}_{j}|{\beta}_{j} \in\mathbb{N}\}$, the co-equal number $\tau$ is defined by
	\begin{equation}
		\tau\left(\eta_1,\eta_2\right)=\sum_{i} \sum_{j}\delta^{{\beta}_{j}}_{\alpha_i}.
	\end{equation}
\end{definition}

\begin{theorem}\label{th5} For any sorted diagonal matrix $\Lambda= {\rm diag}[\lambda_{1}, \cdot\cdot\cdot, \lambda_{n}] \in SPD\left(n\right)$, Wasserstein sectional curvature satisfies
	\begin{equation}
		K|_\Lambda\left(S_1,S_2\right) =
		\frac{3\left(1+\delta^{rt}\right)\lambda_{p}\lambda_{t}}{\left(\lambda_{p}+\lambda_{r}\right)\left(\lambda_{r}+\lambda_{t}\right)\left(\lambda_{p}+\lambda_{t}\right)}.
	\end{equation}
\end{theorem}
\begin{proof}
	Firstly, consider the matrix multiplication
	\begin{equation*}
		\left(S_1 S_2\right)_{ij} = \sum^{n}_{k=1} S^{p,q}_{ik} S^{r,t}_ {kj}=\delta^{p}_i\delta^{qt}\delta^{r}_j+\delta^{p}_i\delta^{qr}\delta^{t}_j+\delta^{q}_i\delta^{pt}\delta^{r}_j+\delta^{q}_i\delta^{pr}\delta^{t}_j.
	\end{equation*}
	Subsequently, combining it with Algorithm 1, we have
	\begin{equation*}
		E_{S_1} E_{S_2} \neq 0 \Longleftrightarrow  S_1 S_2 \neq 0 \Longleftrightarrow \tau\left(\{p,q\},\{r,t\}\right)\geq1,
	\end{equation*}
	where the $E_{S}=[\frac{S_{ij}}{\lambda_{i}+\lambda{j}}]=\Gamma_\Lambda S$, since $\Lambda$ is diagonal (recalling Algorithm 1 for details).\par
	According to the anti-symmetry of curvature tensor, we get
	\begin{equation*}
		\{p,q\}=\{r,t\} \Longrightarrow R|_\Lambda\left(S_1 ,S_2 ,S_1 ,S _2 \right) = 0.
	\end{equation*}\par
	Thus, without loss of generality, the non-vanishing curvature appears only in the particular case:
	\begin{equation*}
		p \neq q = r,p \neq t.
	\end{equation*}\par
	In this case, we have
	\begin{equation*}
		{\left(S_1S_2\right)}_{ij} = \left(1+\delta^{rt}\right)\delta^{p}_i\delta^{t}_j =
		\begin{cases}
			\delta^{p}_i\delta^{t}_j , & \mbox{if } r \neq t,\\
			2\delta^{p}_i\delta^{t}_j , & \mbox{if } r = t,
		\end{cases}
	\end{equation*}
	and
	\begin{equation}\label{eqn43}
		\left(E_{S_1}E_{S_2}\right)_{ij} = \frac{1+\delta^{rt}}{\left(\lambda_{p}+\lambda_{r}\right)\left(\lambda_{r}+\lambda_{t}\right)}\delta^p_i\delta^t_j.
	\end{equation}\par
	Then, we obtain
	\begin{equation*}
		\left(\Gamma_{\Lambda}[E_{S_1},E_{S_2}]\right)_{ij}  = \frac{\left(1+\delta^{rt}\right){\left(\lambda_{p}+\lambda_{t}\right)}^{-1}}{\left(\lambda_{p}+\lambda_{r}\right)\left(\lambda_{r}+\lambda_{t}\right)}\left(\delta^p_i\delta^t_j-\delta^p_j\delta^t_i\right) 
	\end{equation*}
	and
	\begin{equation*}
		\begin{split}
			& \sum_{k=1}^{n}\left(\Gamma_{\Lambda}[E_{S_1},E_{S_2}]\right)_{ik} \left(\Lambda E_{S_2}E_{S_1}\Lambda\right)_{kj} \\
			=\ & \sum_{k=1}^{n}\frac{\left(1+\delta^{rt}\right){\left(\lambda_{p}+\lambda_{t}\right)}^{-1}}{\left(\lambda_{p}+\lambda_{r}\right)\left(\lambda_{r}+\lambda_{t}\right)}\left(\delta^p_i\delta^t_k-\delta^p_k\delta^t_i\right) \frac{\lambda_{p}\lambda_{t}\left(1+\delta^{rt}\right)}{\left(\lambda_{p}+\lambda_{r}\right)\left(\lambda_{r}+\lambda_{t}\right)}\delta^t_k\delta^p_j
			\\
			=\ & \frac{\left(1+\delta^{rt}\right)^2\lambda_{p}\lambda_{t}}{\left(\lambda_{p}+\lambda_{r}\right)^2\left(\lambda_{r}+\lambda_{t}\right)^2\left(\lambda_{p}+\lambda_{t}\right)}\delta^p_i\delta^p_j.
			\\
		\end{split}
	\end{equation*}
	Thus we gain the curvature tensor
	\begin{equation}\label{Ks1s2}
		\begin{split}
			R|_\Lambda\left(S_1 ,S_2 ,S_1 ,S _2 \right) & = 3{\rm tr}\left(E_{S_1}\Lambda\Gamma_{\Lambda}[E_{S_1},E_{S_2}]\Lambda E_{S_2}\right) \\
			& = \frac{3\left(1+\delta^{rt}\right)^2\lambda_{p}\lambda_{t}}{\left(\lambda_{p}+\lambda_{r}\right)^2\left(\lambda_{r}+\lambda_{t}\right)^2\left(\lambda_{p}+\lambda_{t}\right)}.
		\end{split}
	\end{equation}\par
	On the other hand, by \eqref{eqn43} we have
	\begin{equation*}
		\ S_1 \neq S_2 \Longrightarrow \langle S_1,S_2\rangle={\rm tr}\left(E_{S_1}\Lambda E_{S_2}\right) = 0,
	\end{equation*}
	and
	\begin{equation*}
		\langle S_2,S_2\rangle=\frac{1}{2}{\rm tr}\left(E_{S_2}S_2\right) = \frac{1+\delta^{rt}}{\left(\lambda_{r}+\lambda_{t}\right)}.
	\end{equation*}
	Then, we can get the area of sectional parallelogram  as
	\begin{equation}\label{squresection}
		\langle S_1,S_1\rangle \langle S_2,S_2\rangle - \langle S_1,S_2\rangle^2 = \frac{1+\delta^{rt}}{\left(\lambda_{p}+\lambda_{r}\right)\left(\lambda_{r}+\lambda_{t}\right)}.
	\end{equation}\par
	From (\ref{Ks1s2}) and (\ref{squresection}), the proof is done consequently.
\end{proof} \par
With the above expansion for sectional curvatures, we can easily gain the estimate. In fact, any sectional curvature at a fixed point ${A} \in SPD\left(n\right)$ can be controlled by the minimal eigenvalue $\lambda_{min}$ of ${A}$.

\begin{theorem}
	For any ${A} \in SPD\left(n\right)$, there exists an orthonormal basis $\{e_k\}$ for $T_{A}SPD\left(n\right)$ such that for any $e_{k_1},e_{k_2} \in \{e_k\}$
	\begin{equation}
		0<K_{A}\left(e_{k_1},e_{k_2}\right)< \frac{3}{\lambda_{min}\left({A}\right)}.
	\end{equation}
\end{theorem}
\begin{proof}
	As {A} is diagonizable, there exists a $\Lambda \in SPD\left(n\right)$, and ${X}',{Y}' \in T_\Lambda SPD\left(n\right)$ such that $K_{A}\left({X},{Y}\right) = K_\Lambda\left({X}',{Y}'\right)$ from Theorem \ref{DCHEN}. Therefore, by Theorem \ref{th5}, we have
	\begin{equation*}
		K\left(S^{p,q},S^{r,t}\right) =
		\begin{cases}
			\frac{3\lambda_{p}\lambda_{t}}{\left(\lambda_{p}+\lambda_{r}\right)\left(\lambda_{r}+\lambda_{t}\right)\left(\lambda_{p}+\lambda_{t}\right)}, & p\neq q= r\neq t,\ p\neq t ,\\
			\frac{3\lambda_{p}}{\left(\lambda_{p}+\lambda_{t}\right)^2}, & p\neq q= r= t,\\
			0, & else.
		\end{cases}
	\end{equation*}\par
	From the positive definiteness of ${A}$, we see that
	\begin{equation*}
		\lambda_{i}> 0 \Longrightarrow K\left(S^{p,q},S^{r,t}\right)<\frac{3}{\lambda_{p}+\lambda_{t}},
	\end{equation*}
	which gives the desired estimate.
\end{proof} \par
It is remarkable that, sectional curvatures are controlled by the secondly minimal eigenvalue, which implies that the curvature will seldom explode even on a domain almost degenerated. Only when the matrices degenerate at over two dimensions will the curvatures be very large. This phoneme ensures the curvature information make sense in most applications. Some examples for this phoneme can be observed in the next subsection.

\subsection{Scalar Curvature}
In the last part of this section, we calculate the scaler curvature directly.
\begin{theorem} For any ${A} \in SPD\left(n\right)$, the scalar curvature satisfies
	\begin{equation}\label{Curve scalar}
		\rho\left({A}\right) = 3{\rm tr}\left(2\Lambda U U^{T}+\Lambda U^{T}U+\Lambda U\left(U+U^{T}\right)\Lambda \left(U+U^{T}\right)\right),
	\end{equation}
	where $\Lambda= {\rm diag}[\lambda_{1}, \cdot\cdot\cdot, \lambda_{n}]$ is a diagonalization of ${A}$, and $U:=[\frac{1}{\lambda_{i}+\lambda_{j}}]_{i<j}$ is a upper triangle matrix. The subscripts of eigenvalue $\lambda$ must keep consistent when constructing $\Lambda$ and $U$.
\end{theorem}
\begin{proof}
	By symmetry, $\rho(A) = \rho(\Gamma)$ we get
	\begin{equation*}
		\begin{split}
			\rho\left(\Lambda\right)&= \sum_{p=1}^{n}\sum_{r=p}^n\sum_{t=1}^{n}K\left(S^{p,r},S^{r,t}\right) \\
			& =\sum_{p=1}^{n}\left(\sum_{r=p+1}^n\sum_{t=1}^{n,t\neq p}K\left(S^{p,r},S^{r,t}\right)+\sum_{t=1}^{n,t\neq p}K\left(S^{t,p},S^{p,p}\right)\right)\\
			&= 3\sum^{n}_{p=1}\left(\sum^{n}_{r=p+1}\sum^{n,t\neq p}_{t=1}\frac{\left(1+\delta^{rt}\right)\lambda_{p}\lambda_{t}}{\left(\lambda_{p}+\lambda_{r}\right)\left(\lambda_{r}+\lambda_{t}\right)\left(\lambda_{p}+\lambda_{t}\right)}+\sum^{n,t\neq p}_{t=1}\frac{\lambda_{t}}{\left(\lambda_{t}+\lambda_{p}\right)^2}\right)\\
			&= 3\sum^{n}_{p=1}\left(\sum^{n}_{r=p}\sum^{n}_{t=1}\frac{\left(1-\delta^{pr}\right)\left(1+\delta^{rt}\right)\left(1-\delta^{pt}\right)\lambda_{p}\lambda_{t}}{\left(\lambda_{p}+\lambda_{r}\right)
				\left(\lambda_{r}+\lambda_{t}\right)\left(\lambda_{p}+\lambda_{t}\right)}+\sum^{n}_{t=1}\frac{\left(1-\delta^{pt}\right)\lambda_{t}}{\left(\lambda_{t}+\lambda_{p}\right)^2}\right)\\
			&= 3\sum^{n}_{p=1}\sum^{n}_{r=p}\frac{\lambda_{p}\left(1+\delta^{pr}\right)}{\lambda_{p}+\lambda_{r}}\sum^{n}_{t=1}\frac{1+\delta^{rt}}{\lambda_{r}+
				\lambda_{t}}\frac{\lambda_{t}\left(1-\delta^{tp}\right)}{\lambda_{t}+\lambda_{p}}.
		\end{split}
	\end{equation*}\par
	Using the definitions before, we see that $\rho\left(\Lambda\right)$ can be rewritten as the desired trace.
\end{proof}
\par
Figure \ref{figure4} presents some examples for scalar curvatures on $\left(SPD\left(2\right),g_W\right)$, which shows our argument
in the last part of section 6.2.
\begin{figure}[H]
	\centering
	\includegraphics[scale=0.4]{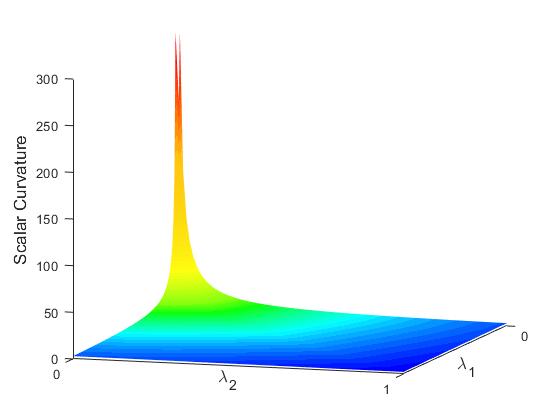}
	\vspace{0pt}
	\caption{Scalar curvatures on SPD(2) }\label{figure4}
\end{figure}

\section{Conclusion and Further Work}
In this paper, we study some geometric characteristics of $\left(SPD\left(n\right),g_W\right)$, including geodesics, the connection, Jacobi fields and curvatures. From this viewpoint, we have gained a basic comprehension for the Wasserstein geometry on
$SPD\left(n\right)$, which shows that the Wasserstein distance has both deep application potentials and the mathematical elegance.
While there still are some open problems left, such as the exact degree of symmetry and more precious bounds of curvatures, some ideas remain to be tested and generalized.
Based on the results in the paper, we have designed a series of algorithms for data processing, and further relevant works are being carried out. The new results will be presented in our future works.
\par \
\par  \
\par
\noindent
\textbf{Acknowledgments} \par
The authors would like to express their sincere thanks to Professor Chao Qian for his valuable suggestions. This subject is supported partially by National Key Research and Development Plan of China (No. 2020YFC2006201).

\end{document}